\documentclass[reqno,a4paper]{amsart}
\usepackage{amssymb}
\usepackage{amsmath}
\begin{document} 
\newtheorem{prop}{Proposition}[section]
\newtheorem{Def}{Definition}[section] \newtheorem{theorem}{Theorem}[section]
\newtheorem{lemma}{Lemma}[section] \newtheorem{Cor}{Corollary}[section]

\title[Unconditional uniqueness for MKG]{\bf Unconditional global well-posedness in energy space for the Maxwell-Klein-Gordon system in temporal gauge}
\author[Hartmut Pecher]{
{\bf Hartmut Pecher}\\
Fachbereich Mathematik und Naturwissenschaften\\
Bergische Universit\"at Wuppertal\\
Gau{\ss}str.  20\\
42119 Wuppertal\\
Germany\\
e-mail {\tt pecher@math.uni-wuppertal.de}}
\date{}

\begin{abstract}
The Maxwell-Klein-Gordon system in temporal gauge is unconditionally globally well-posed in energy space, especially uniqueness holds in the natural solution space. This improves earlier results where uniqueness was only shown in a suitable subspace. It is also locally well-posed for large data below energy space.
\end{abstract}
\maketitle
\renewcommand{\thefootnote}{\fnsymbol{footnote}}
\footnotetext{\hspace{-1.5em}{\it 2010 Mathematics Subject Classification:} 
35Q40, 35L70 }
\normalsize 
\setcounter{section}{0}
\section{Introduction and main results}
\noindent Consider the Maxwell-Klein-Gordon equations
\begin{align}
\label{1.1}
\partial^{\alpha} F_{\alpha \beta} & = Im(\phi \overline{D_{\beta} \phi}) \\
\label{1.2}
D^{\mu}D_{\mu} \phi &= 0
\end{align}
in Minkowski space $\mathbb{R}^{1+3} = \mathbb{R}_t \times \mathbb{R}^3_x$ with metric $diag(-1,1,1,1)$. Greek indices run over $\{0,1,2,3\}$, Latin indices over $\{1,2,3\}$, and the usual summation convention is used.  Here
$$ \phi: \mathbb{R}\times \mathbb{R}^3 \to \mathbb{C} \, , \, A_{\alpha}: \mathbb{R} \times \mathbb{R}^3 \to \mathbb{R} \, , \, F_{\alpha \beta} = \partial_{\alpha} A_{\beta} - \partial_{\beta} A_{\alpha} \, , \, D_{\mu} = \partial_{\mu} - iA_{\mu} \, . $$
$A_{\mu}$ are the gauge potentials, $F_{\mu \nu}$ is the curvature. We use the notation $\partial_{\mu} = \frac{\partial}{\partial x_{\mu}}$, where we write $(x^0,x^1,x^2,x^3)=(t,x^1,x^2,x^3)$ and also $\partial_0 = \partial_t$.

Setting $\beta =0$ in (\ref{1.1}) we obtain the Gauss-law constraint
\begin{equation}
\label{1.3}
\partial^j F_{j 0} = Im(\phi \overline{D_0 \phi})  \, .
\end{equation}
The energy conservation law reads as follows:
\begin{equation}
\label{1.4}
E(t) = \frac{1}{2} \int_{\mathbb{R}^3} \big( \sum_{i=1}^3 F^2_{0i}(t,x) + \sum_{i<j,j=1}^3 F^2_{ij}(t,x) + \sum_{\mu=0}^3 |D_{\mu} \phi(t,x)|^2 \big) dx = E(0) \, .
\end{equation}
The system (\ref{1.1}),(\ref{1.2}) is invariant under the gauge transformations
$$A_{\mu} \to A'_{\mu} = A_{\mu} + \partial_{\mu} \chi \, , \, \phi \to \phi' = e^{i \chi} \phi \, , \, D_{\mu} \to D'_{\mu} = \partial_{\mu} - i A'_{\mu} \, .$$
This allows to impose an additional gauge condition. We exclusively consider the temporal gauge
\begin{equation}
\label{1.5}
A_0 = 0 \, .
\end{equation}
Under this gauge the system (\ref{1.1}),(\ref{1.2}) is given by
\begin{align}
\label{1.6}
\partial_t \partial^j A_j &= - Im(\phi \overline{\partial_t \phi}) \\
\Box A_j - \partial_j(\partial^k A_k) & = Im(\phi \overline{\partial_j \phi} + i A_j |\phi|^2) \\
\Box \phi &= i (\partial^k A_k) \phi + 2i A^k \partial_k \phi + A^k A_k \phi  \, .
\end{align}
Other choices of the gauge are the Coulomb gauge condition $\partial^j A_j =0$ and the Lorenz gauge condition $\partial^{\mu} A_{\mu} = 0$. 
The classical Maxwell-Klein-Gordon system has been studied by Klainerman and Machedon \cite{KM} where the existence of global solutions for data in energy space and above in Coulomb gauge was shown. Uniqueness in a suitable subspace was also shown. For the temporal gauge they also showed a similar result by using a suitable gauge transformation applied to the solution constructed in Coulomb gauge. They made use of a null structure for the main bilinear term to achieve this result. Global well-posedness below energy space was shown by Keel, Roy and Tao \cite{KRT} and local well-posedness almost down to the critical regularity with respect to scaling by Machedon and Sterbenz \cite{MS}. The problem in Lorenz gauge was considered by Selberg and Tesfahun \cite{ST}, who detected a null structure also in this case, and proved global well-posedness in energy space, especially also unconditional uniqueness in this space. Local well-posedness below energy space in Bourgain-Klainerman-Machedon spaces of $X^{s,b}$-type was shown by the author \cite{P}.

The problem in temporal gauge was treated by Yuan \cite{Y} directly in 
$X^{s,b}$-spaces. He stated local well-posedness in $X^{s,b}$-spaces for large data below energy norm, where he just referred to the estimates given for Tao's small data local well-posedness results \cite{T1} in the Yang-Mills case. As a consequence  he
proved existence of a global solution in energy space and also uniqueness in subspaces of $X^{s,b}$-type. Unconditional uniqueness in the natural solution spaces remained open. 

We prove unconditional global well-posedness in energy space and above in the temporal gauge (Theorem \ref{Theorem}). We make use of Tao's estimates \cite{T1} and Yuan's result \cite{Y}. We first detail Yuan's proof of the local existence and uniqueness result for data below energy space. Yuan combined this with energy conservation to achieve the global well-posedness result in $X^{s,b}$-spaces (Theorem \ref{Theorem2}). In the second part we therefore concentrate on the unconditional uniqueness result in the natural solution spaces. We show that any finite energy solution belongs to those spaces of $X^{s,b}$-type below energy norm where uniqueness was shown in the step before (Proposition \ref{Prop.1}). Of course we also need the null structure of some of the nonlinearities, the bilinear estimates for wave-Sobolev spaces $X^{s,b}_{|\tau|=|\xi|}$ by d'Ancona, Foschi and Selberg \cite{AFS}, and Tao's hybrid estimates \cite{T1} for the product of functions in wave-Sobolev spaces $X^{s,b}_{|\tau|=|\xi|}$ and in product Sobolev spaces $X^{s,b}_{\tau = 0}$ (cf. the definition of the spaces below).

Our main result is the following:
\begin{theorem}
\label{Theorem}
Assume $s \ge 1$. For $j=1,2,3$ let $a_j \in H^s({\mathbb R}^3)$ , $a_j' \in H^{s-1}({\mathbb R}^3)$,   $\phi_0 \in H^s({\mathbb R}^3)$ , $\phi_1 \in H^{s-1}({\mathbb R}^3)$ be given satisfying $\partial^j a_j' = - Im(\phi_0 \overline{\phi_1})$ . The Maxwell-Klein-Gordon system (\ref{1.1}),(\ref{1.2}) under the temporal gauge condition $A_0 = 0$
has an (unconditionally) unique global solution
\begin{equation}
\label{1.15}
\phi \in C^0({\mathbb R},H^s({\mathbb R}^3)) \cap C^1({\mathbb R},H^{s-1}({\mathbb R}^3)) \, , \, A_j \in C^0({\mathbb R},H^s({\mathbb R}^3)) \cap C^1({\mathbb R},H^{s-1}({\mathbb R}^3))
\end{equation}
satisfying the initial conditions
\begin{equation}
\label{1.16}
\phi(0) = \phi_0 \, , \, \partial_t \phi(0) = \phi_1 \, , \, A_j(0) = a_j \, \, \partial_t A_j(0) = a_j' \, . 
\end{equation}
\end{theorem}
\noindent
{\bf Remark:} It is very likely that a similar unconditional local well-posedness result holds even below energy norm, i.e. for $s$ less than but close enough to 1, because there is some space in most of the estimates, but we do not persue this further.

We denote the Fourier transform with respect to space and time and with respect to space by $\,\widehat{}\,$ and ${\mathcal F}$, respectively. The operator
$|\nabla|^{\alpha}$ is defined by $(|\nabla|^{\alpha} f)(\xi) = |\xi|^{\alpha} ({\mathcal F}f)(\xi)$ and similarly $ \langle \nabla \rangle^{\alpha}$. $\Box = \partial_t^2 - \Delta$ is the d'Alembert operator. The inhomogeneous Sobolev spaces are denoted by $H^{s,p}$. For $p=2$ we simply denote them by $H^s$. We repeatedly use the Sobolev embeddings $H^{s,p} \subset L^q$ for  $1<p\le q < \infty$ and $\frac{1}{q} \ge \frac{1}{p}-\frac{s}{2}$. \\
$a+ := a + \epsilon$ for a sufficiently small $\epsilon >0$ , so that $a<a+<a++$ , and similarly $a--<a-<a$ , and $\langle \cdot \rangle := (1+|\cdot|^2)^{\frac{1}{2}}$ .

The standard spaces $X^{s,b}_{\pm}$ of Bourgain-Klainerman-Machedon type (which were already considered by M. Beals \cite{B}) belonging to the half waves are the completion of the Schwarz space  $\mathcal{S}({\mathbb R}^4)$ with respect to the norm
$$ \|u\|_{X^{s,b}_{\pm}} = \| \langle \xi \rangle^s \langle  \tau \pm |\xi| \rangle^b \widehat{u}(\tau,\xi) \|_{L^2_{\tau \xi}} \, . $$ 
Similarly we define the wave-Sobolev spaces $X^{s,b}_{|\tau|=|\xi|}$ with norm
$$ \|u\|_{X^{s,b}_{|\tau|=|\xi|}} = \| \langle \xi \rangle^s \langle  |\tau| - |\xi| \rangle^b \widehat{u}(\tau,\xi) \|_{L^2_{\tau \xi}}  $$ and also $X^{s,b}_{\tau =0}$ with norm 
$$\|u\|_{X^{s,b}_{\tau=0}} = \| \langle \xi \rangle^s \langle  \tau  \rangle^b \widehat{u}(\tau,\xi) \|_{L^2_{\tau \xi}} \, .$$
We also define $X^{s,b}_{\pm}[0,T]$ as the space of the restrictions of functions in $X^{s,b}_{\pm}$ to $[0,T] \times \mathbb{R}^3$ and similarly $X^{s,b}_{|\tau| = |\xi|}[0,T]$ and $X^{s,b}_{\tau =0}[0,T]$. We frequently use the estimates $\|u\|_{X^{s,b}_{\pm}} \le \|u\|_{X^{s,b}_{|\tau|=|\xi|}}$ for $b \le 0$ and the reverse estimate for $b \ge 0$.

\section{Preparations and reformulation of the problem}
We decompose $A=(A_1,A_2,A_3)$ into its divergence-free part $A^{df}$ and its curl-free part $A^{cf}$ :
\begin{equation}
\label{1.9}
A = A^{df} + A^{cf} \, ,
\end{equation}
where
\begin{equation}
\label{1.10}
 A^{df} = (-\Delta)^{-1} curl \,curl\, A \quad , \quad A^{cf} = - (-\Delta)^{-1} \nabla\, div \,A \, .
 \end{equation}
Let $P = (-\Delta)^{-1} curl \,curl$ denote the projection operator onto the divergence free part. Then we obtain the equivalent system
\begin{align}
\label{1.11}
\partial_t A^{cf} &= - (-\Delta)^{-1} \nabla Im(\phi \overline{\partial_t \phi}) \\
\label{1.12}
\Box A^{df} & = -P( Im(\phi \overline{ \nabla \phi}) + iA |\phi|^2) \\
\label{1.13}
\Box \phi & = i(\partial^j A_j^{cf}) \phi + 2i A^{df}_j \partial^j \phi +2i A_j^{cf} \partial^j \phi + A^j A_j \phi \, ,
\end{align}
where $A$ is replaced by (\ref{1.9}).

Klainerman and Machedon showed that $A^{df} \cdot \nabla \phi$ and $P(\phi \overline{\nabla \phi})_k$ are null forms. An elementary calculation namely shows that
\begin{align}
\label{1.13'}
2 A^{df}_i \partial^i \phi & = Q_{ij}(\phi,|\nabla|^{-1}(R^i A^j- R^j A^i)) \\
\label{1.12'}
P(\phi \overline{\nabla \phi})_k & = -2i R^j |\nabla|^{-1} Q_{kj}(Re \phi, Im \phi) \, ,
\end{align}
where the null form $Q_{ij}$ is defined by
$$ Q_{ij}(u,v):= \partial_i u \partial_j v - \partial_j u \partial_i v $$
and the Riesz transform by $R_j := |\nabla|^{-1} \partial_j$.

For our further considerations it is also important that the gauge invariance allows to assume besides (\ref{1.5}) the assumption
\begin{equation}
\label{1.14}
A^{cf}(0) = A^{cf}(0,x) = 0 \, .
\end{equation}
One only has to choose 
\begin{equation}
\label{GT}
 \chi(x) = (-\Delta)^{-1} div\, A(0,x) \, .
 \end{equation}
This implies 
$ A_0' = A_0 =0$ and $$A'^{cf}(0)= A^{cf}(0) + \nabla \chi= -(-\Delta)^{-1} \nabla\, div\, A(0) +  (-\Delta)^{-1} \nabla \,div\, A(0) = 0 \, . $$

Defining
\begin{align*}
\phi_{\pm} = \frac{1}{2}(\phi \pm i \langle \nabla \rangle^{-1} \partial_t \phi)&
 \Longleftrightarrow \phi=\phi_+ + \phi_- \, , \, \partial_t \phi = i \langle \nabla \rangle (\phi_+ - \phi_-) \\
 A^{df}_{\pm} = \frac{1}{2}(A^{df} \pm i \langle \nabla \rangle^{-1} \partial_t A^{df}) & \Longleftrightarrow A^{df} = A^{df}_+ + A_-^{df} \, , \, \partial_t A^{df} = i \langle \nabla \rangle(A^{df}_+ - A^{df}_-)
 \end{align*}
 we can rewrite (\ref{1.11}),(\ref{1.12}),(\ref{1.13}) as
 \begin{align}
 \label{1.11*}
 \partial_t A^{cf} &= - (-\Delta)^{-1} \nabla Im(\phi \overline{\partial_t \phi}) \\
 \label{1.12*}
(i \partial_t \pm \langle \nabla \rangle)A_{\pm} ^{df} & = \pm 2^{-1} \langle \nabla \rangle^{-1} ( R.H.S. \, of \, (\ref{1.12}) + A^{df}) \\
\label{1.13*}
(i \partial_t \pm \langle \nabla \rangle) \phi_{\pm} &= \pm 2^{-1} \langle \nabla \rangle^{-1}( R.H.S. \, of \, (\ref{1.13}) + \phi) \, .
\end{align}
The initial data are transformed as follows:
\begin{align}
\label{1.14*}
\phi_{\pm}(0) &= \frac{1}{2}(\phi(0) \pm i^{-1} \langle \nabla \rangle^{-1} \partial_t \phi(0)) \\
\label{1.15*}
A^{df}_\pm(0) & = \frac{1}{2}(A^{df}(0) \pm i^{-1} \langle \nabla \rangle^{-1} \partial_t A^{df}(0) \, .
\end{align}

Very recently the following theorem was claimed by Yuan \cite{Y}. We decided to give a more detailed proof of part 1 of it in the following.
\begin{theorem}
\label{Theorem2}
1. Let $s> \frac{3}{4}$. For $j=1,2,3$ let $a_j \in H^s({\mathbb R}^3)$ , $a_j' \in H^{s-1}({\mathbb R}^3)$,  $\phi_0 \in H^s({\mathbb R}^3)$ , $\phi_1 \in H^{s-1}({\mathbb R}^3)$ be given satisfying $\partial^j a_j' = - Im(\phi_0 \overline{\phi_1})$ . Then there exists $T>0$ such that (\ref{1.11}),(\ref{1.12}),(\ref{1.13}) with initial conditions 
$$\phi(0)=\phi_0 \, , \, \partial_t \phi(0) = \phi_1 \, , \, A^{df}(0) = a^{df} := (-\Delta)^{-1} curl \,curl \,a \, ,$$ 
$$ \partial_t A^{df}(0) = a'^{df} := (-\Delta)^{-1} curl\, curl \,a' \, , \, A^{cf}(0) = a^{cf}:= -(-\Delta)^{-1} \nabla\, div \,a $$
such that
\begin{equation}
\label{*}
a^{cf}=0 \, ,
\end{equation}
where $a:=(a_1,a_2,a_3)$ , $a':=(a_1',a_2',a_3')$, has a unique local solution such that
$$ \phi_{\pm} \in X_{\pm}^{s,\frac{3}{4}+}[0,T] \, , \, A^{df}_{\pm} \in X^{s,\frac{3}{4}+}[0,T] \ , \, A^{cf} \in X^{s+\frac{1}{4},\frac{1}{2}+}_{\tau=0}[0,T] \, . $$
This solution satisfies $\phi=\phi_+ + \phi_- \in C^0([0,T],H^s) \cap C^1([0,T],H^{s-1}) $ , $A= A^{df}_+ + A^{df}_- + A^{cf}\in C^0([0,T],H^s) \cap C^1([0,T],H^{s-1})$.
\\
2. If $ s \ge 1$ one obtains after performing a gauge transform
a solution of (\ref{1.1}),(\ref{1.2}),(\ref{1.16}) in temporal gauge $A_0=0$ with
$$ \phi \, , \, A \in C^0([0,T],H^s({\mathbb R}^3)) \cap C^1([0,T],H^{s-1}({\mathbb R}^3))  $$ without the assumption (\ref{*}), which exists globally, i.e. T can be chosen arbitrarily.
\end{theorem}

Fundamental for us are the following bilinear estimates in wave-Sobolev spaces which were proven by d'Ancona, Foschi and Selberg in the three-dimensional case $n=3$ in \cite{AFS} in a more general form which include many limit cases which we do not need.
\begin{prop}
\label{Prop.2}
Let $n=3$. The estimate
$$\|uv\|_{X_{|\tau|=|\xi|}^{-s_0,-b_0}} \lesssim \|u\|_{X^{s_1,b_1}_{|\tau|=|\xi|}} \|v\|_{X^{s_2,b_2}_{|\tau|=|\xi|}} $$ 
holds, provided the following conditions hold:
\begin{align*}
\nonumber
& b_0 + b_1 + b_2 > \frac{1}{2} \\
\nonumber
& b_0 + b_1 > 0 \\
\nonumber
& b_0 + b_2 > 0 \\
\nonumber
& b_1 + b_2 > 0 \\
\nonumber
&s_0+s_1+s_2 > 2 -(b_0+b_1+b_2) \\
\nonumber
&s_0+s_1+s_2 > \frac{3}{2} -\min(b_0+b_1,b_0+b_2,b_1+b_2) \\
\nonumber
&s_0+s_1+s_2 > 1 - \min(b_0,b_1,b_2) \\
\nonumber
&s_0+s_1+s_2 > 1 \\
 &(s_0 + b_0) +2s_1 + 2s_2 > \frac{3}{2} \\
\nonumber
&2s_0+(s_1+b_1)+2s_2 > \frac{3}{2} \\
\nonumber
&2s_0+2s_1+(s_2+b_2) > \frac{3}{2} \\
\nonumber
&s_1 + s_2 \ge \max(0,-b_0) \\
\nonumber
&s_0 + s_2 \ge \max(0,-b_1) \\
\nonumber
&s_0 + s_1 \ge \max(0,-b_2)   \, .
\end{align*}
\end{prop}
We also need an immediate consequence of Strichartz' estimate for the homogeneous wave equation.
\begin{lemma}
\label{Lemma1}
Assume $2  < q \le \infty$ , $2 \le r < \infty$ with $ \frac{1}{2} \le \frac{1}{q} + \frac{1}{r} \le 1$. Then
\begin{equation}
\nonumber
\|u\|_{L^q_t L^r_x} \lesssim \| |\nabla|^{1-\frac{2}{r}} u\|_{X^{0,1-(\frac{1}{q}+\frac{1}{r})+}_{|\tau|=|\xi|}} \, . 
\end{equation}
\end{lemma}
\begin{proof}[Proof: (\cite{ST}, Lemma 2.1)]
Interpolate the Strichartz' type estimate (combine \cite{P1}, Theorem 1 with the transfer principle)
$$ \|u\|_{L_t^{\frac{2}{1-\epsilon}} L_x^{\frac{2}{\epsilon}}} \lesssim \| |\nabla|^{1-\epsilon} u \|_{X^{0,\frac{1}{2}+}_{|\tau|=|\xi|}} $$
for $0 < \epsilon \le 1$ with the trivial identity $\|u\|_{L^2_t L^2_x} = \|u\|_{X^{0,0}_{|\tau|=|\xi|}}$.
\end{proof}
A very powerful nontrivial variant with $L^r_x L^q_t$-norms rather than $L^q_t L^r_x$-norms was detected by Tao (cf. \cite{T1}, Prop. 4.1) (see also \cite{KMBT}, appendix by D. Tataru): 
\begin{lemma}
\label{Lemma2}
The following estimate holds :
\begin{equation}
\label{Tao}
 \|u\|_{ L^4_x L^2_t} \lesssim \|u\|_{X^{\frac{1}{4},\frac{1}{2}+}_{|\tau|=|\xi|}} \, . 
 \end{equation}
\end{lemma}
\begin{proof}[Proof of Theorem \ref{Theorem2}]
For part 2. which relies on the energy conservation law we refer to Yuan \cite{Y} and
remark that although it was formulated by Yuan only for $s=1$ it in fact holds for any $s\ge 1$, because one notices that the special gauge transform (\ref{GT}) preserves also higher regularity.\\
Proof of part 1: 
Using (\ref{1.13'}),(\ref{1.12'}),(\ref{1.14}),(\ref{1.11*})-(\ref{1.15*}) by standard arguments the local existence and uniqueness proof for large data is reduced to the following estimates:
\begin{equation}
\label{N1}
\| |\nabla|^{-1} (\phi_1 \partial_t \phi_2)\|_{X^{s+\frac{1}{4},-\frac{1}{2}+2\epsilon-}_{\tau=0}} \lesssim \|\phi_1\|_{X^{s,\frac{3}{4}+\epsilon}_{|\tau|=|\xi|}} \|\phi_2\|_{X^{s,\frac{3}{4}+\epsilon}_{|\tau|=|\xi|}} \, ,
\end{equation}
\begin{align}
\label{N2}
&\|Q_{ij}(|\nabla|^{-1}\phi_1,\phi_2)\|_{X^{s-1,-\frac{1}{4}+2\epsilon}_{\tau|=|\xi|}} + \||\nabla|^{-1}Q_{ij}(\phi_1,\phi_2)\|_{X^{s-1,-\frac{1}{4}+2\epsilon}_{\tau|=|\xi|}} \\
\nonumber
& \hspace{15em} \lesssim \|\phi_1\|_{X^{s,\frac{3}{4}+\epsilon}_{|\tau|=|\xi|}} \|\phi_2\|_{X^{s,\frac{3}{4}+\epsilon}_{|\tau|=|\xi|}} \, ,
\end{align}
\begin{equation}
\label{N3}
\| \nabla A \phi \|_{X^{s-1,-\frac{1}{4}+2\epsilon}_{\tau|=|\xi|}} +
\| A \nabla \phi \|_{X^{s-1,-\frac{1}{4}+2\epsilon}_{\tau|=|\xi|}} \lesssim \|A\|_{X^{s+\frac{1}{4},\frac{1}{2}+\epsilon}_{\tau =0}}  \|\phi\|_{X^{s,\frac{3}{4}+\epsilon}_{|\tau|=|\xi|}} \, ,
\end{equation}
\begin{equation}
\label{N4}
\| A_1 A_2 A_3 \|_{X^{s-1,-\frac{1}{4}+2\epsilon}_{\tau|=|\xi|}} \lesssim \prod_{i=1}^3 \min(\|A_i\|_{X^{s,\frac{3}{4}+\epsilon}_{|\tau|=|\xi|}},\|A_i\|_{X^{s+\frac{1}{4},\frac{1}{2}+\epsilon}_{\tau =0}} ) \, .
\end{equation}
We especially remark that in order to obtain a large data result we need on the left hand side $X^{s,b}$-spaces with $b=-\frac{1}{2}+2\epsilon-$ and $b=-\frac{1}{4}+2\epsilon$ instead of $b=-\frac{1}{2}+\epsilon$ and $b=-\frac{1}{4}+\epsilon$, respectively. This can not be achieved for additional estimates which are needed in the Yang-Mills case. This is why one only obtains a small data result by this method in the paper by Tao \cite{T1}  (cf. especially the footnote in the introduction of his paper), which nevertheless is fundamental for our proof. It is simple to adapt Tao's proof in order to obtain (\ref{N2}),(\ref{N3}) and (\ref{N4}). Therefore we concentrate on the proof of (\ref{N1}), where the time derivative requires a modification of it. Similarly as Tao we start with\\
{\bf Claim 1:} 
\begin{equation*}
\| |\nabla|^{-1} (\phi_1 \partial_t \phi_2)\|_{X^{s+\frac{1}{4},-\frac{1}{2}+2\epsilon-}_{\tau=0}} \lesssim \|\phi_1\|_{X^{s+\frac{1}{4}-\epsilon,\frac{1}{2}+\epsilon}_{\tau=0}} \|\phi_2\|_{X^{s+\frac{1}{4}-\epsilon,\frac{1}{2}+\epsilon}_{\tau=0}}  \, .
\end{equation*}
As usual the regularity of $|\nabla|^{-1}$ is harmless in three dimensions (\cite{T}, Cor. 8.2) and it can be replaced by $\langle \nabla \rangle^{-1}$. Taking care of the time derivative we reduce to
\begin{align*}
\big|\int \int u_1 u_2 u_3 dx dt\big| \lesssim \|u_1\|_{X^{s+\frac{1}{4}-\epsilon,\frac{1}{2}+\epsilon}_{\tau =0}}
\|u_2\|_{X^{s+\frac{1}{4}-\epsilon,-\frac{1}{2}+\epsilon}_{\tau =0}}
\|u_3\|_{X^{\frac{3}{4}-s,\frac{1}{2}-2\epsilon+}_{\tau =0}} \, ,
\end{align*}
which follows from Sobolev's multiplication rule.\\
{\bf Claim 2:} 
\begin{align*}
&\| |\nabla|^{-1} (\phi_1 \partial_t \phi_2)\|_{X^{s+\frac{1}{4},-\frac{1}{2}+2\epsilon-\delta}_{\tau=0}} + \| |\nabla|^{-1} (\phi_2 \partial_t \phi_1)\|_{X^{s+\frac{1}{4},-\frac{1}{2}+2\epsilon-\delta}_{\tau=0}} \\
& \hspace{20em}\lesssim \|\phi_1\|_{X^{s,\frac{3}{4}+\epsilon}_{|\tau|=|\xi|}} \|\phi_2\|_{X^{s+\frac{1}{4}-\epsilon,\frac{1}{2}+2\epsilon-\delta}_{\tau=0}} 
\end{align*}
for $0<\delta \ll \epsilon$. \\
a. If $\widehat{\phi}$ is supported in  $ ||\tau|-|\xi|| \gtrsim |\xi| $ we remark that
$$ \|\phi\|_{X^{s+\frac{1}{4}-\epsilon,\frac{1}{2}+\epsilon}_{\tau=0}} \lesssim \|\phi\|_{X^{s,\frac{3}{4}}_{|\tau|=|\xi|}} \,, $$
so that claim 2 follows from claim 1.\\
b. It remains to show
$$ \big|\int\int (uv_t w + uvw_t) dxdt \big| \lesssim 
\|u\|_{X^{\frac{3}{4}-\epsilon,\frac{1}{2}-2\epsilon+\delta}_{\tau =0}}
\|w\|_{X^{s,\frac{3}{4}+\epsilon}_{|\tau| =|\xi|}}
\|v\|_{X^{s+\frac{1}{4}-\epsilon,\frac{1}{2}+2\epsilon-\delta}_{\tau =0}} \, $$
whenever $\widehat w$ is supported in $||\tau|-|\xi|| \ll |\xi|$.
This is equivalent to
$$ \int_* m(\xi_1,\xi_2,\xi_3,\tau_1,\tau_2,\tau_3) \prod_{i=1}^3 \widehat{u}_i(\xi_i,\tau_i) d\xi d\tau \lesssim \prod_{i=1}^3 \|u_i\|_{L^2_{xt}} \, $$
where $d\xi = d\xi_1 d\xi_2 d\xi_3$ , $d\tau = d\tau_1 d\tau_2 d\tau_3$ and * denotes integration over $\sum_{i=1}^3 \xi_i = \sum_{i=1}^3 \tau_i = 0$. The Fourier transforms are nonnegative without loss of generality. Here
$$ m= \frac{(|\tau_2|+|\tau_3|) \chi_{||\tau_3|-|\xi_3|| \ll |\xi_3|}}{\langle \xi_1 \rangle^{\frac{3}{4}-s} \langle \tau_1 \rangle^{\frac{1}{2}-2\epsilon+\delta} \langle \xi_2 \rangle^{s+\frac{1}{4}-\epsilon} \langle \tau_2 \rangle^{\frac{1}{2}+2\epsilon-\delta} \langle \xi_3 \rangle^s \langle |\tau_3|-|\xi_3|\rangle^{\frac{3}{4}+\epsilon}} \, . $$
Since $\langle \tau_3 \rangle \sim \langle \xi_3 \rangle$ and $\tau_1+\tau_2+\tau_3=0$ we have 
\begin{equation}
\label{N4'}
|\tau_2| + |\tau_3| \lesssim \langle \tau_1 \rangle^{\frac{1}{2}-2\epsilon+\delta} \langle \tau_2 \rangle^{\frac{1}{2}+2\epsilon-\delta} +\langle \tau_1 \rangle^{\frac{1}{2}-2\epsilon+\delta} \langle \xi_3 \rangle^{\frac{1}{2}+2\epsilon-\delta} +\langle \tau_2 \rangle^{\frac{1}{2}+2\epsilon-\delta} \langle \xi_3 \rangle^{\frac{1}{2}-2\epsilon+\delta} , 
\end{equation}
so that concerning the first term on the right hand side of (\ref{N4'}) we have to show
$$\big|\int\int uvw dx dt\big| \lesssim \|u\|_{X^{\frac{3}{4}-\epsilon,0}_{\tau=0}} |v\|_{X^{s+\frac{1}{4}-\epsilon,0}_{\tau=0}} |w\|_{X^{s,\frac{3}{4}+\epsilon}_{|\tau|=|\xi|}} \ , $$
which easily follows from Sobolev's multiplication rule.\\
Concerning the second term on the right hand side of (\ref{N4'}) we use $\langle \xi_1 \rangle^{s-\frac{3}{4}} \lesssim \langle \xi_2 \rangle^{s-\frac{3}{4}} + \langle \xi_3 \rangle^{s-\frac{3}{4}}$, so that we reduce to
\begin{align*}
&\big|\int\int uvw dx dt\big|  \lesssim \|u\|_{L^2_x L^2_t}  \|v\|_{L^{\frac{4}{1-4(2\epsilon - \delta)}}_x L^{\infty}_t} \|w\|_{L^{\frac{4}{1+4(2\epsilon-\delta)}}_x L^2_t} \\
& \lesssim\|u\|_{X^{0,0}_{\tau=0}} (\|v\|_{X^{1-\epsilon,\frac{1}{2}+2\epsilon-\delta}_{\tau=0}} \|w\|_{X^{s-\frac{1}{2}-2\epsilon+\delta,\frac{3}{4}+\epsilon}_{|\tau|=|\xi|}} 
+ \|v\|_{X^{s+\frac{1}{4}-\epsilon,\frac{1}{2}+2\epsilon-\delta}_{\tau=0}} \|w\|_{X^{\frac{1}{4}-2\epsilon+\delta,\frac{3}{4}+\epsilon}_{|\tau|=|\xi|}} )\,, 
\end{align*}
where we used Sobolev and interpolated between (\ref{Tao}) and the trivial identity $\|w\|_{L^2_x L^2_t} = \|w\|_{X^{0,0}_{|\tau|=|\xi|}}$ to obtain
$$\|w\|_{L^{\frac{4}{1+4(2\epsilon-\delta)}}_x L^2_t} \lesssim  \|w\|_{X^{\frac{1}{4}-2\epsilon+\delta,\frac{1}{2}+}_{|\tau|=|\xi|}}\, .$$
Concerning the last term on the right hand side of (\ref{N4'}) we can similarly reduce to
\begin{align*}
&\big|\int\int uvw dx dt\big| 
 \lesssim \|u\|_{L^2_x L^{\frac{1}{2\epsilon-\delta}}_t}  
\|v\|_{L^4_x L^2_t} \|w\|_{L^4_x L^{\frac{2}{1-2(2\epsilon-\delta)}}_t} \\
& \lesssim \|u\|_{X^{0,\frac{1}{2}-2\epsilon+\delta}_{\tau=0}} 
(\|v\|_{X^{1-\epsilon,0}_{\tau=0}}
\|w\|_{X^{s-\frac{1}{2}+2\epsilon-\delta,\frac{3}{4}+\epsilon}_{|\tau|=|\xi|}} 
+ \|v\|_{X^{s+\frac{1}{4}-\epsilon,0}_{\tau=0}} \|w\|_{X^{\frac{1}{4}+2\epsilon-\delta,\frac{3}{4}+\epsilon}_{|\tau|=|\xi|}} )\,, 
\end{align*}
where we interpolated between (\ref{Tao}) and Strichartz' estimate $\|u\|_{L^4_x L^4_t} \lesssim \|u\|_{X^{\frac{1}{2},\frac{1}{2}+}_{|\tau|0|\xi|}}$, which gives
$$ \|u\|_{L^4_x L^{\frac{2}{1-2(2\epsilon-\delta)}}_t} \lesssim \|u\|_{X^{\frac{1}{4}+2\epsilon-\delta,\frac{1}{2}+}_{|\tau|=|\xi|}} \, .$$
Claim 2 is now proven.

We now come to the proof of (\ref{N1}). \\
If $\widehat{\phi}$ is supported in $||\tau|-|\xi|| \gtrsim |\xi|$ we obtain
$$\|\phi\|_{X^{s+\frac{1}{4}-\epsilon,\frac{1}{2}+2\epsilon-\delta}_{\tau =0}}
\lesssim \|\phi\|_{X^{s,\frac{3}{4}+\epsilon - \delta}_{|\tau|=|\xi|}} \, . $$
which implies that (\ref{N1}) follows from claim 2, if $\widehat{\phi}_1$ or $\widehat{\phi}_2$ have this support property. So we may assume that both functions are supported in $||\tau|-|\xi|| \ll |\xi|$. This means that it suffices to show
$$ \int_* m(\xi_1,\xi_2,\xi_3,\tau_1,\tau_2,\tau_3) \prod_{i=1}^3 \widehat{u}_i(\xi_i,\tau_i) d\xi d\tau \lesssim \prod_{i=1}^3 \|u_i\|_{L^2_{xt}} \, , $$
where
$$m= \frac{|\tau_3|\chi_{||\tau_2|-|\xi_2|| \ll |\xi_2|} \chi_{||\tau_3|-|\xi_3|| \ll |\xi_3|}}{\langle \xi_1 \rangle^{\frac{3}{4}-s} \langle \tau_1 \rangle^{\frac{1}{2}-2\epsilon} \langle \xi_2 \rangle^s \langle |\tau_2|-|\xi_2| \rangle^{\frac{3}{4}+\epsilon} \langle \xi_3 \rangle^s \langle |\tau_3|-|\xi_3|\rangle^{\frac{3}{4}+\epsilon}} \, . $$
Since $\langle \tau_3 \rangle \sim \langle \xi_3 \rangle$ , $\langle \tau_2 \rangle \sim \langle \xi_2 \rangle$ and $\tau_1+\tau_2+\tau_3=0$ we have 
\begin{equation}
|\tau_3| \lesssim \langle \tau_1 \rangle^{\frac{1}{2}-2\epsilon} \langle \xi_3 \rangle^{\frac{1}{2}+2\epsilon} +\langle \xi_2 \rangle^{\frac{1}{2}-2\epsilon} \langle \xi_3 \rangle^{\frac{1}{2}+2\epsilon} , 
\end{equation}
The first term on the right hand side is treated by Prop. \ref{Prop.2} which gives
$$\big|\int \int uvw dx dt\big| \lesssim \|u\|_{X^{\frac{3}{4}-s,0}_{\tau=0}} \|v\|_{X^{s,\frac{3}{4}+\epsilon}_{|\tau|=|\xi|}} \|w\|_{X^{s-\frac{1}{2}-2\epsilon,\frac{3}{4}+\epsilon}_{|\tau|=|\xi|}} $$
for $ s > \frac{3}{4}$ and $\epsilon >0$ sufficiently small. In order to treat the second term on the right hand side we use $\langle \xi_1 \rangle^{s-\frac{3}{4}} \lesssim \langle \xi_2 \rangle^{s-\frac{3}{4}} + \langle \xi_3 \rangle^{s-\frac{3}{4}}$. We have to show
\begin{align}
\label{N5}
&\big|\int \int uvw dx dt\big|\\
\nonumber
 & \hspace{2em}\lesssim \|u\|_{X^{0,\frac{1}{2}-2\epsilon}_{\tau=0}} (\|v\|_{X^{\frac{1}{4}+2\epsilon,\frac{3}{4}+\epsilon}_{|\tau|=|\xi|}} \|w\|_{^{s-\frac{1}{2}-2\epsilon,\frac{3}{4}+\epsilon}_{|\tau|=|\xi|}} 
 + \|v\|_{X^{s-\frac{1}{2}+2\epsilon,\frac{3}{4}+\epsilon}_{|\tau|=|\xi|}} \|w\|_{^{\frac{1}{4}-2\epsilon,\frac{3}{4}+\epsilon}_{|\tau|=|\xi|}}) \, .
\end{align}
This follows from
\begin{align*}
&\big|\int \int uvw dx dt\big|\\
& \lesssim \|u\|_{L^2_x L^{\frac{1}{2\epsilon}}_t} \|v\|_{L^4_x L^2_t} \|w\|_{L^4_x L^{\frac{2}{1-4\epsilon}}_t} + \|u\|_{L^2_x L^{\frac{1}{7\epsilon}}_t} \|v\|_{L^{\frac{4}{1-12\epsilon}}_x L^{\frac{2}{1-12\epsilon}}_t} \|w\|_{L^{\frac{4}{1+12\epsilon}}_x L^{\frac{2}{1-2\epsilon}}_t} \,,
\end{align*}
which we now prove.
By Sobolev we namely have
$$ \|u\|_{L^2_x L^{\frac{1}{7\epsilon}}_t} + \|u\|_{L^2_x L^{\frac{1}{2\epsilon}}_t}
\lesssim \|u\|_{X^{0,\frac{1}{2}-2\epsilon}_{\tau =0}}$$
and by (\ref{Tao}) $$\|v\|_{ L^4_x L^2_t} \lesssim \|v\|_{X^{\frac{1}{4},\frac{1}{2}+}_{|\tau|=|\xi|}}\, .$$
We obtain by interpolation between (\ref{Tao}) and Strichartz' estimate $\|w\|_{L^4_x L^4_t} \lesssim \|w\|_{X^{\frac{1}{2},\frac{1}{2}+}_{|\tau|=|\xi|}}$ with interpolation parameter $\theta = 1-8\epsilon $ the estimate
$$\|w\|_{L^4_x L^{\frac{2}{1-4\epsilon}}_t} \lesssim \|w\|_{X^{\frac{1}{4}+2\epsilon,\frac{1}{2}+}_{|\tau|=|\xi|}} \lesssim \|w\|_{X^{s-\frac{1}{2}-2\epsilon,\frac{1}{2}+}_{|\tau=|\xi|}} $$
for $s> \frac{3}{4}$ and $\epsilon >0$ sufficiently small. 
Next we interpolate between (\ref{Tao}) and the estimate
$$\|v\|_{L^{\infty}_x L^{\infty}_t} = \|v\|_{L^{\infty}_t L^{\infty}_x} \lesssim \|v\|_{L^{\infty}_t H^{\frac{3}{2}+}_x} \lesssim \|v\|_{X^{\frac{3}{2}+,\frac{1}{2}+}_{|\tau|=|\xi|}} $$
with interpolation parameter $\theta = 1-12\epsilon$ and obtain for $s>\frac{3}{4}$ and $\epsilon>0$ sufficiently small
$$ \|v\|_{L^{\frac{4}{1-12\epsilon}}_x L^{\frac{2}{1-12\epsilon}}_t} \lesssim \|v\|_{X^{\frac{1}{4}+15\epsilon+,\frac{1}{2}+}_{|\tau||\xi|}} \lesssim \|v\|_{X^{s-\frac{1}{2}+2\epsilon,\frac{1}{2}+}_{|\tau|=|\xi|}} \, . $$
Interpolation between Strichartz' estimate and the trivial identity $\|w\|_{L^2_x L^2_t} = \|w\|_{X^{0,0}_{|\xi|=|\tau|}}$ gives $\|w\|_{L^{\frac{16}{7}}_x L^{\frac{16}{7}}_t} \lesssim \|w\|_{X^{\frac{1}{8},\frac{1}{8}+}_{|\tau|=|\xi|}}$. Finally interpolating this again with (\ref{Tao}) with interpolation parameter $\theta=1-16\epsilon$ we arrive at
$$\|w\|_{L^{\frac{4}{1+12\epsilon}}_x L^{\frac{2}{1-2\epsilon}}_t} \lesssim \|w\|_{X^{\frac{1}{4}-2\epsilon,\frac{1}{2}+}_{|\tau|=|\xi|}} \, .$$
Summarizing all these results we obtain (\ref{N5}).
\end{proof}

For the existence part of Theorem \ref{Theorem} we rely on this result. We remark that any solution of (1.1),(1.2) in temporal gauge $A_0=0$ with $\phi,A \in C^0([0,T],H^1) \cap C^1([0,T],L^2)$ is via a gauge transformation equivalent to a solution of (\ref{1.11}),(\ref{1.12}) and (\ref{1.13}), which fulfills $A^{cf}(0)=0$, with the same regularity. Using also the uniqueness in the spaces of $X^{s,b}$-type which appear in Theorem \ref{Theorem2} we easily see that our Theorem \ref{Theorem} is a direct consequence of the following Proposition.
\begin{prop}
\label{Prop.1}
Let $T>0$. Assume $(\phi,A^{df},A^{cf})$ is a solution of (\ref{1.11}),(\ref{1.12}),(\ref{1.13}) with initial conditions
$$\phi(0)=\phi_0 \in H^1 \, , \, \partial_t\phi(0) = \phi_1 \in L^2 \, , \, A^{df}(0)= (-\Delta)^{-1}curl\, curl \,a \in H^1 \, , $$ $$ \partial_t A^{df}(0) = (-\Delta)^{-1}curl\, curl \,a' \in L^2 \, , \, A^{cf}(0) = 0 $$
and
$$\phi \in C^0([0,T],H^1) \cap C^1([0,T],L^2) \, , \, A \in C^0([0,T],H^1) \cap C^1([0,T],L^2) \, . $$
Then
$$ \phi_{\pm} \in X_{\pm}^{\frac{3}{4}+,\frac{3}{4}+}[0,T] \, , \, A^{df}_{\pm} \in X_{\pm}^{\frac{3}{4}+,\frac{3}{4}+}[0,T] \, , \, A^{cf} \in X^{1+,\frac{1}{2}+}_{\tau =0}[0,T] \, . $$
\end{prop}

\section{Proof of Prop. \ref{Prop.1}}
We use the equivalent system (\ref{1.11*}),(\ref{1.12*}),(\ref{1.13*}) with initial conditions (\ref{1.14*}),(\ref{1.15*}) and $A^{cf}(0) =0$. We drop $[0,T]$ from spaces of the type $X^{s,b}[0,T]$.\\
{\bf Claim 1:} $ A^{df}_{\pm} \in X^{\frac{2}{3},\frac{5}{6}-}_{\pm}$ \\
From (\ref{1.12*}) we obtain
\begin{align*}
\|A^{df}_{\pm}\|_{X^{\frac{2}{3},\frac{5}{6}--}_{\pm}} \lesssim \|A^{df}_{\pm}(0)\|_{H^{\frac{2}{3}}} + \|\phi \overline{\nabla \phi} \|_{X^{-\frac{1}{3},-\frac{1}{6}-}_{\pm}} &+ \| A |\phi|^2 \|_{X^{-\frac{1}{3},-\frac{1}{6}-}_{\pm}} \\ & + \sum_{\pm} \|A_{\pm}\|_{X^{-\frac{1}{3},-\frac{1}{6}-}_{\pm}}  . 
\end{align*}
By Lemma \ref{Lemma1} with $r=3$ and $q=2+$ we have
$$\|u\|_{L^{2+}_t L^3_x} \lesssim \| |\nabla|^{\frac{1}{3}} u\|_{X^{0,\frac{1}{6}+}_{|\tau|=|\xi|}} \, . $$
Hence by duality
$$ \|\phi \overline{\nabla \phi} \|_{X^{-\frac{1}{3},-\frac{1}{6}-}_{\pm}} \hspace{-0.1em}\lesssim \|\phi \overline{\nabla \phi}\|_{L^{2-}_t L^{\frac{3}{2}}_x} \hspace{-0.1em} \lesssim \|\phi\|_{L^{\infty}_t L^6_x} \|\nabla \phi\|_{L^{\infty}_t L^2_x} T^{\frac{1}{2}+} \hspace{-0.1em}\lesssim \|\phi\|_{L^{\infty}_t H^1_x}^2 T^{\frac{1}{2}+} \hspace{-0.1em}< \hspace{-0.1em}\infty \,.  $$
Moreover
\begin{align*}
\|A|\phi|^2\|_{X^{-\frac{1}{3},-\frac{1}{6}-}_{\pm}} & \lesssim \| A|\phi|^2\|_{L^{2-}_t L^{\frac{3}{2}}_x} \lesssim \|A\|_{L^{\infty}_t L^{\frac{9}{2}}_x} \|\phi\|^2_{L^{\infty}_t L^{\frac{9}{2}}_x} T^{\frac{1}{2}+} \\
& \lesssim \| A \|_{L^{\infty}_t H^1_x} \|\phi\|^2_{L^{\infty}_t H^1_x} T^{\frac{1}{2}+} < \infty \, .
\end{align*}
The linear term is easily estimated by $T^{\frac{1}{2}} \|A\|_{L^{\infty}_t L^2_x}$.\\
{\bf Claim 2:} $\phi_{\pm} \in X^{\frac{2}{3},\frac{5}{6}-}_{\pm}$ \\
From (\ref{1.13*}) we obtain
\begin{align*}
\|\phi_{\pm}\|_{X^{\frac{2}{3},\frac{5}{6}--}_{\pm}} & \lesssim  \|\phi_{\pm}(0)\|_{H^{\frac{2}{3}}} + \|div \,A \,\phi \|_{X^{-\frac{1}{3},-\frac{1}{6}-}_{\pm}} + \|A  \cdot \nabla \phi \|_{X^{-\frac{1}{3},-\frac{1}{6}-}_{\pm}} \\
& \quad + \|A_j A^j \phi \|_{X^{-\frac{1}{3},-\frac{1}{6}-}_{\pm}} + \sum_{\pm} \|\phi_{\pm}\|_{X^{-\frac{1}{3},-\frac{1}{6}-}_{\pm}} \, .
\end{align*}
Because $A$ and $\phi$ have the same regularity all the terms can be treated exactly as in claim 1. \\
{\bf Claim 3:} $A^{df}_{\pm} \in X^{\frac{4}{5}-,\frac{1}{2}+}_{\pm}$ , $ \phi_{\pm} \in X^{\frac{4}{5}-,\frac{1}{2}+}_{\pm}$. \\
This follows by interpolation from $A^{df}_{\pm}, \phi_{\pm} \in X^{\frac{2}{3},\frac{5}{6}-}_{\pm} \cap X^{1,0}_{\pm}$. \\
{\bf Claim 4:} $A^{cf} \in X^{1+,\frac{1}{2}+}_{\tau =0}$ \\
By (\ref{1.11*}) and claim 3 we only have to prove (recalling $A^{cf}(0)=0$):
\begin{equation}
\label{1.17}
\| |\nabla|^{-1} (\phi \overline{\partial_t \phi}) \|_{X^{1+,-\frac{1}{2}+}_{\tau=0}} \lesssim \|\phi\|^2_{X^{\frac{4}{5}-,\frac{1}{2}+}_{|\tau|=|\xi|}}
\end{equation}
By \cite{T}, Cor. 8.2 we can replace the left hand side by $\|\phi \overline{\partial_t \phi}\|_{X^{0+,-\frac{1}{2}+}}$ (the singularity of $|\nabla|^{-1}$ is harmless in three dimensions). (\ref{1.17}) is equivalent to
\begin{align}
\label{1.18}
\int_* \frac{\widehat{u}_1(\xi_1,\tau_1)}{\langle |\xi_1| - |\tau_1| \rangle^{\frac{1}{2}+} \langle \xi_1 \rangle^{\frac{4}{5}-}} \frac{|\tau_2| \widehat{u}_2(\xi_2,\tau_2)}{\langle |\xi_2| - |\tau_2| \rangle^{\frac{1}{2}+} \langle \xi_2 \rangle^{\frac{4}{5}-}} \frac{\widehat{u}_3(\xi_3,\tau_3)\langle \xi_3 \rangle^{0+}}{\langle \tau_3 \rangle^{\frac{1}{2}-}} d \xi d \tau \lesssim \prod_{i=1}^3 \|u_i\|_{L^2_{tx}} \, ,
\end{align}
where $d\xi = d\xi_1 d\xi_2 d\xi_3$ , $d\tau = d\tau_1 d\tau_2 d\tau_3$ and * denotes integration over $\sum_{i=1}^3 \xi_i = \sum_{i=1}^3 \tau_i = 0$. The Fourier transforms are nonnegative without loss of generality.\\
{\bf Case 1:} ( $|\xi_1| \gg |\tau_1|$ or $|\tau_1| \gg |\xi_1|$ ) and ( $|\xi_2| \gg |\tau_2|$ or $|\tau_2| \gg |\xi_2 |$ ). \\
We obtain by Sobolev's multiplication rule
\begin{align*}
\Big|\int \int v_1 \partial_t v_2 v_3 dx dt \Big| & \lesssim \|v_1\|_{H^{\frac{1}{2}+}_t H^{\frac{3}{4}+}_x} \|\partial_t v_2\|_{H^{-\frac{1}{2}+}_t H^{\frac{3}{4}+}_x} \|v_3\|_{H^{\frac{1}{2}-}_t H^{0-}_x} \\
& \lesssim \|v_1\|_{X^{\frac{3}{4}+,\frac{1}{2}+}_{\tau =0}} \|v_2\|_{X^{\frac{3}{4}+,\frac{1}{2}+}_{\tau =0}} \|v_3\|_{X^{0-,\frac{1}{2}-}_{\tau =0}} \, .
\end{align*}
Our frequency assumptions imply for $j=1$ and $j=2$:
$$ \langle \tau_j \rangle^{\frac{1}{2}+} \langle \xi_j \rangle^{\frac{3}{4}+} \lesssim \langle \xi_j \rangle^{\frac{3}{4}+} \langle |\tau_j| - |\xi_j| \rangle^{\frac{1}{2}+} \, , $$
so that
$$ \|v_j\|_{X^{\frac{3}{4}+,\frac{1}{2}+}_{\tau =0}} \lesssim \|v_j\|_{X^{\frac{3}{4}+,\frac{1}{2}+}_{|\tau|=|\xi|}} \, . $$
This implies (\ref{1.17}) (with $\frac{4}{5}-$ replaced by $\frac{3}{4}+$ even). \\
{\bf Case 2:} $|\xi_2| \sim |\tau_2|$ and ($|\xi_1| \gg |\tau_1|$ or $|\xi_1| \ll |\tau_1|$).\\
Using $\tau_1+\tau_2+\tau_3=0$ one easily checks
\begin{equation}
\label{1.18'}
|\tau_1| + |\tau_2| \lesssim \langle \tau_1 \rangle^{\frac{1}{2}+} \langle \tau_3 \rangle^{\frac{1}{2}-} + \langle \tau_1 \rangle^{\frac{1}{2}-} \langle \xi_2 \rangle^{\frac{1}{2}+} + \langle \tau_3 \rangle^{\frac{1}{2}-} \langle \xi_2 \rangle^{\frac{1}{2}-} \, .
\end{equation}
We want to show
\begin{equation}
\label{1.19}
\| \phi_1 \overline{\partial_t \phi_2}\|_{X^{0+,-\frac{1}{2}+}_{\tau =0}} \lesssim \|\phi_1\|_{X^{\frac{4}{5}-,\frac{1}{2}+}_{\tau =0}} \|\phi_2\|_{X^{\frac{4}{5}-,\frac{1}{2}+}_{|\tau| = |\xi|}} \, , 
\end{equation}
which implies (\ref{1.17}) by the estimate
$\|\phi_1\|_{X^{\frac{4}{5}-,\frac{1}{2}+}_{\tau=0}} \lesssim \|\phi_1\|_{X^{\frac{4}{5}-,\frac{1}{2}+}_{|\tau|=|\xi|}} \, . $
(\ref{1.19}) is equivalent to
\begin{align}
\label{1.20}
\int_* \frac{\widehat{u}_1(\xi_1,\tau_1)}{\langle \tau_1 \rangle^{\frac{1}{2}+} \langle \xi_1 \rangle^{\frac{4}{5}-}} \frac{|\tau_2| \widehat{u}_2(\xi_2,\tau_2)}{\langle |\xi_2| - |\tau_2| \rangle^{\frac{1}{2}+} \langle \xi_2 \rangle^{\frac{4}{5}-}} \frac{\widehat{u}_3(\xi_3,\tau_3)\langle \xi_3 \rangle^{0+}}{\langle \tau_3 \rangle^{\frac{1}{2}-}} d \xi d \tau \lesssim \prod_{i=1}^3 \|u_i\|_{L^2_{tx}} \, .
\end{align}
Using (\ref{1.18'}) and $\langle \xi_3 \rangle^{0+} \lesssim \langle \xi_1 \rangle^{0+} + \langle \xi_2 \rangle^{0+}$ we have to show the following three estimates:
\begin{align}
\label{1.21}
\big|\int \int v_1 v_2 v_3 dx dt \big| &\lesssim \|v_1\|_{X^{\frac{4}{5}-,0}_{\tau=0}} \|v_2\|_{X^{\frac{4}{5}-,\frac{1}{2}+}_{|\tau|=|\xi|}} \|v_3\|_{X^{0,0}_{\tau=0}} \\
\label{1.22}
\big|\int \int v_1 v_2 v_3 dx dt \big| &\lesssim \|v_1\|_{X^{\frac{4}{5}-,0+}_{\tau=0}} \|v_2\|_{X^{\frac{3}{10}-,\frac{1}{2}+}_{|\tau|=|\xi|}} \|v_3\|_{X^{0,\frac{1}{2}-}_{\tau=0}} \\
\label{1.23}
\big|\int \int v_1 v_2 v_3 dx dt \big| &\lesssim \|v_1\|_{X^{\frac{4}{5}-,\frac{1}{2}+}_{\tau=0}} \|v_2\|_{X^{\frac{3}{10}-,\frac{1}{2}+}_{|\tau|=|\xi|}} \|v_3\|_{X^{0,0}_{\tau=0}} \, .
\end{align}
(\ref{1.21}) follows by the Sobolev multiplication rule, and (\ref{1.22}) as follows by Lemma \ref{Lemma2}:
\begin{align*}
\big| \int \int v_1 v_2 v_3 dxdt \big| & \lesssim \|v_1\|_{L^4_x L^{2+}_t} \|v_2\|_{L^4_x L^2_t} \|v_3\|_{L^2_x L^{\infty -}_t} \\
& \lesssim \|v_1\|_{H^{\frac{3}{4}}_x H^{0+}_t} \|v_2\|_{X^{\frac{1}{4},\frac{1}{2}+}_{|\tau|=|\xi|}} \|v_3\|_{L^2_x H^{\frac{1}{2}-}_t} \\
& \lesssim \|v_1\|_{X^{\frac{3}{4},0+}_{\tau=0}} 
\|v_2\|_{X^{\frac{1}{4},\frac{1}{2}+}_{|\tau|=|\xi|}}
 \|v_3\|_{X^{0,\frac{1}{2}-}_{\tau=0}} \, .
\end{align*}
Finally (\ref{1.23}) follows from
\begin{align*}
\big| \int \int v_1 v_2 v_3 dxdt \big| & \lesssim \|v_1\|_{L^4_x L^{\infty}_t} \|v_2\|_{L^4_x L^2_t} \|v_3\|_{L^2_x L^2_t} \\
& \lesssim \|v_1\|_{H^{\frac{3}{4}}_x H^{\frac{1}{2}+}_t} \|v_2\|_{X^{\frac{1}{4},\frac{1}{2}+}_{|\tau|=|\xi|}} \|v_3\|_{L^2_x L^2_t} \\
& \lesssim \|v_1\|_{X^{\frac{3}{4},\frac{1}{2}+}_{\tau=0}} 
\|v_2\|_{X^{\frac{1}{4},\frac{1}{2}+}_{|\tau|=|\xi|}}
 \|v_3\|_{X^{0,0}_{\tau=0}} \, .
\end{align*}
{\bf Case 3:} $|\xi_1| \sim |\tau_1|$ and ($|\xi_2| \gg |\tau_2|$ or $|\xi_2| \ll |\tau_2|$).\\
This can be handled exactly as Case 2, because by (\ref{1.18'}) it is irrelevant which of the factors $\phi_1$ or $\phi_2$ contains the time derivative. \\
{\bf Case 4:} $|\xi_2| \sim |\tau_2|$ and $|\xi_1| \sim |\tau_1|$. \\
In this case we use
\begin{align*}
|\tau_2| \lesssim \langle \tau_3 + \tau_1 \rangle^{\frac{1}{2}-} \langle \tau_2 \rangle^{\frac{1}{2}+} \lesssim (\langle \tau_3 \rangle^{\frac{1}{2}-} + \langle \xi_1 \rangle^{\frac{1}{2}++} \langle \tau_1 \rangle^{0--}) \langle \xi_2 \rangle^{\frac{1}{2}+} \, .
\end{align*}
This reduces the desired estimate (\ref{1.18}) to the following two estimates:
\begin{align}
\label{1.25}
\int_* \frac{\widehat{u}_1(\xi_1,\tau_1)}{\langle |\xi_1| - |\tau_1| \rangle^{\frac{1}{2}+} \langle \xi_1 \rangle^{\frac{4}{5}--}} \frac{ \widehat{u}_2(\xi_2,\tau_2)}{\langle |\xi_2| - |\tau_2| \rangle^{\frac{1}{2}+} \langle \xi_2 \rangle^{\frac{3}{10}--}} \widehat{u}_3(\xi_3,\tau_3) d \xi d \tau \lesssim \prod_{i=1}^3 \|u_i\|_{L^2_{tx}} \\
\nonumber
\int_* \frac{\widehat{u}_1(\xi_1,\tau_1)}{\langle |\xi_1| - |\tau_1| \rangle^{\frac{1}{2}+} \langle \xi_1 \rangle^{\frac{3}{10}--} \langle \tau_1 \rangle^{0++}} \frac{ \widehat{u}_2(\xi_2,\tau_2)}{\langle |\xi_2| - |\tau_2| \rangle^{\frac{1}{2}+} \langle \xi_2 \rangle^{\frac{3}{10}--}} \frac{\widehat{u}_3(\xi_3,\tau_3)}{\langle \tau_3 \rangle^{\frac{1}{2}-}} d \xi d \tau \\
\label{1.26}
\lesssim \prod_{i=1}^3 \|u_i\|_{L^2_{tx}} \, .
\end{align}
(\ref{1.25}) is a consequence of Proposition \ref{Prop.2} with the choice $s_0=b_0=0$ , $s_1 = \frac{4}{5}-$ , $s_2 = \frac{3}{10}-$ , $b_1=b_2=\frac{1}{2}+$. \\
(\ref{1.26}) is proven using Lemma \ref{Lemma2} as follows:
\begin{align*}
\big | \int \int v_1 v_2 v_3 dx dt \big| & \lesssim \|v_1\|_{L^4_x L^{2+}_t} \|v_2\|_{L^4_x L^2_t} \|v_3|_{L^2_x L^{\infty -}_t} \\
& \lesssim \|v_1\|_{L^4_x H^{0+}_t} \|v_2\|_{L^4_x L^2_t} \|v_3\|_{L^2_x H^{\frac{1}{2}-}_t} \\
& \lesssim \| \langle \partial_t \rangle^{0+} v_1 \|_{X^{\frac{1}{4},\frac{1}{2}+}_{|\tau|=|\xi|}} \|v_2\|_{X^{\frac{1}{4},\frac{1}{2}+}_{|\tau|=|\xi|}} \|v_3\|_{X^{0,\frac{1}{2}-}_{\tau=0}} \, ,
\end{align*}
where $ (\langle \partial_t \rangle^{0+} v)\widehat{\,} (\tau,\xi)  := \langle \tau \rangle^{0+} \widehat{v}(\tau,\xi) $, from which (\ref{1.26}) follows. \\
{\bf Claim 5:} $A^{df}_{\pm} \in X_{\pm}^{\frac{3}{4}+,\frac{3}{4}+} $. \\ 
{\bf Term 1:}
In view of (\ref{1.12*}), (\ref{1.12'}) and claim 3 we have to prove:
$$ \| |\nabla|^{-1} Q_{ij}(\phi_1,  \phi_2)\|_{X^{-\frac{1}{4}+\epsilon,-\frac{1}{4}+}_{\pm''}} \lesssim \|\phi_1\|_{X^{\frac{4}{5}-,\frac{1}{2}+}_{\pm}} \|\phi_2\|_{X^{\frac{4}{5}-,\frac{1}{2}+}_{\pm'}}\, . $$
Using \cite{T}, Corollary 8.2 and thereby ignoring the singularity of $|\nabla|^{-1}$ we can replace the left hand side by $\|Q_{ij}(\phi_1, \phi_2)\|_{X^{-\frac{5}{4}+\epsilon,-\frac{1}{4}+}_{\pm''}}$ . Thus it suffices to give the following estimate:
\begin{equation}
\label{1.27}
\int_* m(\xi_1,\xi_2,\xi_3,\tau_1,\tau_2,\tau_3) \prod_{i=1}^3 \widehat{u}_i(\xi_i,\tau_i) d\xi d\tau \lesssim \prod_{i=1}^3 \|u_i\|_{L^2_{tx}} \, ,
\end{equation}
where 
\begin{align*}
m &= m(\xi_1,\xi_2,\xi_3,\tau_1,\tau_2,\tau_3) \\
&= \frac{|\xi_1 \times \xi_2|}{\langle -\tau_1 \pm|\xi_1|\rangle^{\frac{1}{2}+} \langle \xi_1 \rangle^{\frac{4}{5}-} \, \langle -\tau_2 \pm' |\xi_2| \rangle^{\frac{1}{2}+} \langle \xi_2 \rangle^{\frac{4}{5}-} \, \langle -\tau_3 \pm'' |\xi_3| \rangle^{\frac{1}{4}-} \langle \xi_3 \rangle^{\frac{5}{4}-\epsilon}} \, .
\end{align*}
We define $N_i := \langle \xi_i \rangle$ and denote by $N_{max}$ , $N_{med}$ and $N_{min}$ the largest, second largest and smallest of them. Using $\xi_1 + \xi_2 + \xi_3 =0$ we have $N_{max} \sim N_{med} \ge N_{min}$.

In the following our main task is to give pointwise bounds for $m$, which allow an application of Prop. \ref{Prop.2}.

We apply \cite{S}, Lemma 2.1 to estimate the angle $\angle(\xi_1,\xi_2)$ between two vectors $\xi_1$ and $\xi_2$ and obtain the following bound for the cross product of these vectors.\\
Case a: ($N_1=N_{max}$ and $N_2=N_{min}$) or ( $N_1 = N_{min}$ and $N_2 = N_{max}$). 
\begin{align}
\nonumber
&|\xi_1 \times \xi_2|  \le |\xi_1||\xi_2| \angle(\pm\xi_1,\pm'\xi_2) \\
\nonumber
&\lesssim |\xi_1||\xi_2| \Big( \frac{\langle -\tau_1\pm |\xi_1|  \rangle^{\frac{1}{2}} + \langle -\tau_2\pm' |\xi_2|  \rangle^{\frac{1}{2}}    }{\min(\langle \xi_1 \rangle,\langle \xi_2 \rangle)^{\frac{1}{2}}} + \frac{\langle -\tau_3 \pm'' |\xi_3|\rangle^{\frac{1}{4}-}}{\min(\langle \xi_1 \rangle,\langle \xi_2 \rangle)^{\frac{1}{4}-}} \Big) \\
\label{1.28}
& \lesssim N_{max} N_{min}^{\frac{1}{2}} (\langle |-\tau_1 \pm |\xi_1|\rangle^{\frac{1}{2}} + (\langle -\tau_2 \pm'|\xi_2|\rangle^{\frac{1}{2}}) + N_{max} N_{min}^{\frac{3}{4}+} \langle -\tau_3 \pm''|\xi_3|\rangle^{\frac{1}{4}-} . 
\end{align}
Case b: $N_1 \sim N_2 \sim N_{max}$ and $N_3 = N_{min}$.\\
By the simple observation
$$|\xi_1 \times \xi_2| = |\xi_1 \times (\xi_1+\xi_2)| = |\xi_1 \times \xi_3| \le |\xi_1||\xi_3| \angle(\pm\xi_1,\pm''\xi_3)$$
we obtain the same bound as in case 1.\\
We consider several cases depending on which of the three terms on the right hand side of (\ref{1.28}) is dominant, estimate $m$ in each of these cases and apply Proposition \ref{Prop.2} to deduce (\ref{1.27}).\\
1. If the last term is dominant, we obtain
$$m \lesssim \frac{N_{max} N_{\min}^{\frac{3}{4}+}}{N_1^{\frac{4}{5}-} N_2^{\frac{4}{5}-} N_3^{\frac{5}{4}-\epsilon}} \frac{1}{\langle|\tau_1|-|\xi_1|\rangle^{\frac{1}{2}+} \langle|\tau_2|-|\xi_2|\rangle^{\frac{1}{2}+}} \, . $$
1.1. In the case $N_1 \sim N_2 \sim N_{max}$ and $N_3 = N_{min}$ we obtain the bound
$$m \lesssim \frac{1}{N_1^{\frac{3}{5}-}  N_3^{\frac{1}{2}-\epsilon}} \frac{1}{\langle|\tau_1|-|\xi_1|\rangle^{\frac{1}{2}+} \langle|\tau_2|-|\xi_2|\rangle^{\frac{1}{2}+}} \, .$$
1.2. If $N_3 \sim N_2 \sim N_{max}$ and $N_1 = N_{min}$ (or similarly $N_3 \sim N_1 \sim N_{max}$ and $N_2=N_{min}$) we obtain the bound
$$ m \lesssim \frac{1}{N_1^{\frac{1}{20}-}  N_3^{\frac{21}{20}-\epsilon}} \frac{1}{\langle|\tau_1|-|\xi_1|\rangle^{\frac{1}{2}+} \langle|\tau_2|-|\xi_2|\rangle^{\frac{1}{2}+}}  \,.$$
2. If the first term is dominant (the case where the second term is dominant can be handled similarly), we obtain
$$m \lesssim \frac{N_{max} N_{\min}^{\frac{1}{2}}}{N_1^{\frac{4}{5}-} N_2^{\frac{4}{5}-} N_3^{\frac{5}{4}-\epsilon}} \frac{1}{\langle|\tau_2|-|\xi_2|\rangle^{\frac{1}{2}+} \langle|\tau_3|-|\xi_3|\rangle^{\frac{1}{4}-}} \, . $$                             2.1 If $N_1 \sim N_2 \sim N_{max}$ and $N_3 = N_{min}$ we have
the bound
$$ m \lesssim \frac{1}{N_1^{\frac{3}{5}-}  N_3^{\frac{3}{4}-\epsilon}} \frac{1}{\langle|\tau_2|-|\xi_2|\rangle^{\frac{1}{2}+} \langle|\tau_3|-|\xi_3|\rangle^{\frac{1}{4}-}} \, .$$          
2.2. If $N_2 \sim N_3 \sim N_{max}$ and $N_1 = N_{min}$
we obtain
$$ m \lesssim \frac{1}{N_3^{\frac{21}{20}-\epsilon} N_1^{\frac{3}{10}-}}\frac{1}{\langle|\tau_2|-|\xi_2|\rangle^{\frac{1}{2}+} \langle|\tau_3|-|\xi_3|\rangle^{\frac{1}{4}-}}  \,.$$   
2.3. If $N_1 \sim N_3 \sim N_{max}$ and $N_2 = N_{min}$
this leads to the bound
$$ m \lesssim \frac{1}{N_1^{\frac{21}{20}-\epsilon} N_2^{\frac{3}{10}-}}\frac{1}{\langle|\tau_2|-|\xi_2|\rangle^{\frac{1}{2}+} \langle|\tau_3|-|\xi_3|\rangle^{\frac{1}{4}-}}  \,.$$  
In each of these cases an application of Prop. \ref{Prop.2} leads to the desired estimate (\ref{1.27}). \\
{\bf Term 2:} Moreover we have to give a bound for the cubic term $\|A |\phi|^2\|_{X^{-\frac{1}{4}+,-\frac{1}{4}+}_{|\tau|=|\xi|}}$. \\
By two applications of Prop. \ref{Prop.2} we obtain
\begin{align}
\nonumber
\|A^{df} |\phi|^2 \|_{X^{-\frac{1}{4}+\epsilon,-\frac{1}{4}+}_{|\tau|=|\xi|}} & \lesssim \| A^{df} \|_{X^{\frac{4}{5}-,\frac{1}{2}+}_{|\tau|=|\xi|}} \| |\phi|^2 \|_{X^{\frac{1}{5}+2\epsilon,0}_{|\tau|=|\xi|}} \\
\label{2.1}
& \lesssim \| A^{df} \|_{X^{\frac{4}{5}-,\frac{1}{2}+}_{|\tau|=|\xi|}} \| \phi \|^2_{X^{\frac{4}{5}-,\frac{1}{2}+}_{|\tau|=|\xi|}} \, .
\end{align}
We also obtain
\begin{align}
\nonumber
\|A^{cf} |\phi|^2 \|_{X^{-\frac{1}{4}+\epsilon,-\frac{1}{4}+}_{|\tau|=|\xi|}} & \lesssim \| A^{cf} \|_{X^{1+,\frac{1}{2}+}_{\tau=0}} \| |\phi|^2 \|_{X^{\frac{1}{4}+\epsilon,\frac{1}{4}+}_{|\tau|=|\xi|}} \\
\label{2.2}
& \lesssim \| A^{cf} \|_{X^{1+,\frac{1}{2}+}_{\tau=0}} \| \phi \|^2_{X^{\frac{4}{5}-,\frac{1}{2}+}_{|\tau|=|\xi|}} \, .
\end{align}
The first estimate holds by Sobolev and the second one by Prop. \ref{Prop.2}. Using claim 3 and claim 4 we have proven claim 5. \\
{\bf Claim 6:} $ \phi_{\pm} \in X^{\frac{3}{4}+,\frac{3}{4}+}_{\pm}$.\\
This is the most difficult case. Several terms have to be estimated in view of (\ref{1.13}).\\
{\bf Term 1:} Using (\ref{1.13'}) we first want to show
$$ \| Q_{ij}(\phi,|\nabla|^{-1} R^i A_{df}^j \|_{X^{-\frac{1}{4}+,-\frac{1}{4}+}_{ \pm''}} \lesssim \|\phi\|_{X^{\frac{4}{5}-,\frac{1}{2}+}_{ \pm}} \|A^{df}\|_{X^{\frac{3}{4}+\epsilon,\frac{3}{4}+}_{ \pm'}} \, ,$$ 
which can be controlled by claim 3 and claim 5 ,
and a similar estimate for $Q_{ij}(\phi,|\nabla|^{-1} R^j A_{df}^i)$. We have to give the following estimate:
\begin{equation}
\label{2.3}
\int_* m(\xi_1,\xi_2,\xi_3,\tau_1,\tau_2,\tau_3) \widehat{u}_1(\xi_1,\tau_1)  \widehat{u}_2(\xi_2,\tau_2) \widehat{u}_3(\xi_3,\tau_3) \,d\xi d\tau \lesssim \prod_{i=1}^3 \|u_i\|_{L^2_{xt}} \, , 
\end{equation}
where
\begin{align*}
m = \frac{\frac{|\xi_1 \times \xi_2|}{|\xi_2|}}{\langle -\tau_1 \pm|\xi_1|\rangle^{\frac{1}{2}+} \langle \xi_1 \rangle^{\frac{4}{5}-} \, \langle -\tau_2 \pm' |\xi_2| \rangle^{\frac{3}{4}+} \langle \xi_2 \rangle^{\frac{3}{4}+\epsilon} \, \langle -\tau_3 \pm'' |\xi_3| \rangle^{\frac{1}{4}-} \langle \xi_3 \rangle^{\frac{1}{4}-\epsilon}} \, .
\end{align*}
By an application of \cite{S}, Lemma 2.1 we obtain as above: \\
a. In the case $N_1 \sim N_{max}$ and $N_2 = N_{min}$ the estimate (\ref{1.28}) implies
\begin{align}
\label{2.4}
\frac{|\xi_1 \times \xi_2|}{|\xi_2|}  
 \lesssim \frac{N_{max}}{ N_{min}^{\frac{1}{2}}} (\langle -\tau_1 \pm|\xi_1|\rangle^{\frac{1}{2}} + (\langle -\tau_2 \pm'|\xi_2|\rangle^{\frac{1}{2}}) +\frac{ N_{max}}{ N_{min}^{\frac{1}{4}-}} \langle -\tau_3 \pm''|\xi_3|\rangle^{\frac{1}{4}-} . 
\end{align} 
b. In the case $N_1 = N_{min}$ and $N_2 \sim N_3 \sim N_{max}$ we similarly obtain the improved bound
\begin{align}
\label{2.5}
\frac{|\xi_1 \times \xi_2|}{|\xi_2|}  
 \lesssim  N_{min}^{\frac{1}{2}} (\langle -\tau_1 \pm|\xi_1|\rangle^{\frac{1}{2}} + (\langle -\tau_2 \pm'|\xi_2|\rangle^{\frac{1}{2}}) + N_{min}^{\frac{3}{4}+} \langle -\tau_3 \pm''|\xi_3|\rangle^{\frac{1}{4}-} . 
\end{align} 
c. In the case $N_1 \sim N_2 \sim N_{max}$ and $N_3 = N_{min}$ we obtain
\begin{align*}
\frac{|\xi_1 \times \xi_2|}{|\xi_2|} = \frac{|\xi_1 \times (\xi_1+\xi_2)|}{|\xi_2|} =
\frac{|\xi_1 \times \xi_3|}{|\xi_2|}\angle(\pm\xi_1,\pm''\xi_3) \sim |\xi_3| \angle(\pm\xi_1,\pm''\xi_3) \, ,
\end{align*}
which implies the same bound as in case b.

We obtain the following bounds for $m$ depending on which of the terms on the right hand side of (\ref{2.4}) and (\ref{2.5}) is dominant.\\
1. If the last term is dominant we obtain
$$m \lesssim \frac{1}{N_1^{\frac{4}{5}-} N_2^{\frac{3}{4}+\epsilon} N_3^{\frac{1}{4}-}} \frac{N_{max}}{N_{min}^{\frac{1}{4}-}} \frac{1}{\langle |\xi_2|-|\tau_2| \rangle^{\frac{3}{4}+} \langle |\xi_1|-|\tau_1| \rangle^{\frac{1}{2}+}} \, . $$
1.1. If $N_1 \sim N_3 \sim N_{max}$ and $N_2 = N_{min}$ we have
$$ m \lesssim \frac{1}{N_{max}^{\frac{1}{20}-} N_{min}^{1+\epsilon-}} \frac{1}{\langle |\xi_2|-|\tau_2| \rangle^{\frac{3}{4}+} 
\langle |\xi_1|-|\tau_1| \rangle^{\frac{1}{2}+}} \, . $$
1.2 If $N_1 \sim N_2 \sim N_{max}$ and $N_3=N_{min}$ we obtain
$$ m \lesssim \frac{1}{N_{max}^{\frac{11}{20}+\epsilon-} N_{min}^{\frac{1}{2}-}} \frac{1}{\langle |\xi_2|-|\tau_2| \rangle^{\frac{3}{4}+} 
\langle |\xi_1|-|\tau_1| \rangle^{\frac{1}{2}+}} \, . $$
1.3 If $N_2 \sim N_3 \sim N_{max}$ and $N_1 = N_{min}$ we obtain the bound
$$ m \lesssim \frac{1}{N_{max}^{\epsilon-} N_{min}^{\frac{21}{20}-}} \frac{1}{\langle |\xi_2|-|\tau_2| \rangle^{\frac{3}{4}+} 
\langle |\xi_1|-|\tau_1| \rangle^{\frac{1}{2}+}} \, . $$
2. If the first term is dominant we get
$$m \lesssim \frac{1}{N_1^{\frac{4}{5}-} N_2^{\frac{3}{4}+\epsilon} N_3^{\frac{1}{4}-}} \frac{N_{max}}{N_{min}^{\frac{1}{2}}} \frac{1}{\langle |\xi_2|-|\tau_2| \rangle^{\frac{3}{4}+} \langle |\xi_3|-|\tau_3| \rangle^{\frac{1}{4}-}} \, . $$
2.1. If $N_1 \sim N_3 \sim N_{max}$ and $N_2 = N_{min}$ we get
$$ m \lesssim \frac{1}{N_{max}^{\frac{1}{20}-} N_{min}^{\frac{5}{4}+\epsilon}} \frac{1}{\langle |\xi_2|-|\tau_2| \rangle^{\frac{3}{4}+} 
\langle |\xi_3|-|\tau_3| \rangle^{\frac{1}{4}-}} \, . $$
2.2. If $N_1 \sim N_2 \sim N_{max}$ and $N_3=N_{min}$ we obtain
$$ m \lesssim \frac{1}{N_{max}^{\frac{11}{20}+\epsilon-} N_{min}^{\frac{3}{4}-}} \frac{1}{\langle |\xi_2|-|\tau_2| \rangle^{\frac{3}{4}+} 
\langle |\xi_3|-|\tau_3| \rangle^{\frac{1}{4}-}} \, . $$
2.3 If $N_2 \sim N_3 \sim N_{max}$ and $N_1 = N_{min}$ we use (\ref{2.5}) to obtain
\begin{align*}
m &\lesssim \frac{1}{N_1^{\frac{4}{5}-} N_2^{\frac{3}{4}+\epsilon} N_3^{\frac{1}{4}-}} N_{min}^{\frac{1}{2}} \frac{1}{\langle |\xi_2|-|\tau_2| \rangle^{\frac{3}{4}+} \langle |\xi_3|-|\tau_3| \rangle^{\frac{1}{4}-}} \\
&\lesssim \frac{1}{N_{max}^{1+\epsilon-} N_{min}^{\frac{3}{10}-}} \frac{1}{\langle |\xi_2|-|\tau_2| \rangle^{\frac{3}{4}+} \langle |\xi_3|-|\tau_3| \rangle^{\frac{1}{4}-}} \, . \end{align*}
3. If the second term is dominant we estimate as follows:
$$m \lesssim \frac{1}{N_1^{\frac{4}{5}-} N_2^{\frac{3}{4}+\epsilon} N_3^{\frac{1}{4}-}} \frac{N_{max}}{N_{min}^{\frac{1}{2}}} \frac{1}{\langle |\xi_1|-|\tau_1| \rangle^{\frac{1}{2}+} \langle |\xi_3|-|\tau_3| \rangle^{\frac{1}{4}-}} \, . $$
3.1. If $N_1 \sim N_3 \sim N_{max}$ and $N_2 = N_{min}$ we get
$$ m \lesssim \frac{1}{N_{max}^{\frac{1}{20}-} N_{min}^{\frac{5}{4}+\epsilon-}} \frac{1}{\langle |\xi_1|-|\tau_1| \rangle^{\frac{1}{2}+} 
\langle |\xi_3|-|\tau_3| \rangle^{\frac{1}{4}-}} \, . $$
3.2. If $N_1 \sim N_2 \sim N_{max}$ and $N_3=N_{min}$ we obtain
$$ m \lesssim \frac{1}{N_{max}^{\frac{11}{20}+\epsilon-} N_{min}^{\frac{3}{4}-}} \frac{1}{\langle |\xi_1|-|\tau_1| \rangle^{\frac{1}{2}+} 
\langle |\xi_3|-|\tau_3| \rangle^{\frac{1}{4}-}} \, . $$
2.3 If $N_2 \sim N_3 \sim N_{max}$ and $N_1 = N_{min}$ we get
$$ m \lesssim \frac{1}{N_{max}^{\epsilon-} N_{min}^{\frac{13}{10}-}} \frac{1}{\langle |\xi_1|-|\tau_1| \rangle^{\frac{1}{2}+} 
\langle |\xi_3|-|\tau_3| \rangle^{\frac{1}{4}-}} \, . $$
In each of these cases an application of Prop. \ref{Prop.2} leads to the desired estimate (\ref{2.3}).\\
{\bf Term 2:} Next we want to give similar estimates for the terms containing $A^{cf}$. For $Q_{ij}(\phi,|\nabla|^{-1} R^i A^j_{cf})$ and  $Q_{ij}(\phi,|\nabla|^{-1} R^j A^i_{cf})$ we no longer rely on the null structure so that it behaves like $\partial_i \phi A^{cf}_j$.
They are treated together with the terms $A^{cf}_i \partial^i \phi$ and $\partial_i A^{cf}_i \phi$. We want to prove for $\epsilon > 0$:
$$ \|A_j^{cf} \partial_i \phi\|_{X^{-\frac{1}{5}-\epsilon,0}_{|\tau|=|\xi|}} + \|\partial_i A^{cf}_i \phi\|_{X^{-\frac{1}{5}-\epsilon,0}_{|\tau|=|\xi|}} \lesssim \|A^{cf}\|_{X^{1+,\frac{1}{2}+}_{\tau =0}} \|\phi\|_{X^{\frac{4}{5}-,\frac{1}{2}+}_{|\tau|=|\xi|}} \, , $$
which in view of claim 3 and claim 4 is more than enough (we could allow the $X^{-\frac{1}{4}+,-\frac{1}{4}+}_{|\tau|=|\xi|}$-norm on the left hand side). These estimates would follow if we prove
$$
\int_* m(\xi_1,\xi_2,\xi_3,\tau_1,\tau_2,\tau_3) \widehat{u}_1(\xi_1,\tau_1)  \widehat{u}_2(\xi_2,\tau_2) \widehat{u}_3(\xi_3,\tau_3) d\xi d\tau \lesssim \prod_{i=1}^3 \|u_i\|_{L^2_{xt}} \, , 
$$
where 
$$ m = \frac{|\xi_2| + |\xi_3|}{\langle \xi_1 \rangle^{\frac{1}{5}+\epsilon} \langle \xi_2 \rangle^{\frac{4}{5}-} \langle |\tau_2| - |\xi_2|\rangle^{\frac{1}{2}+}  \langle \xi_3 \rangle^{1+}\langle \tau_3 \rangle^{\frac{1}{2}+}} \, .$$
The following argument is closely related to the proof of a similar estimate in  \cite{T1}.\\
{\bf Case 1:} $|\xi_2| \le |\xi_1|$ ($\Rightarrow$ $|\xi_2|+|\xi_3| \lesssim |\xi_1|$ due to $\xi_1 + \xi_2 + \xi_3 =0$.)
This implies
$$ m \lesssim \frac{\langle \xi_1 \rangle^{\frac{4}{5}-\epsilon}}{ \langle \xi_2 \rangle^{\frac{4}{5}-} \langle |\tau_2| - |\xi_2|\rangle^{\frac{1}{2}+}  \langle \xi_3 \rangle^{1+}\langle \tau_3 \rangle^{\frac{1}{2}+}} \, .$$
By two applications of the averaging principle (\cite{T}, Prop. 5.1) we may replace $m$ by
$$ m' = \frac{\langle \xi_1 \rangle^{\frac{4}{5}-\epsilon} \chi_{||\tau_2|-|\xi_2||\sim 1} \chi_{|\tau_3| \sim 1}}{\langle \xi_2 \rangle^{\frac{4}{5}-} \langle \xi_3 \rangle^{1+}} \, . $$
Let now $\tau_2$ be restricted to the region $\tau_2 =T + O(1)$ for some integer $T$. Then $\tau_1$ is restricted to $\tau_1 = -T + O(1)$, because $\tau_1 + \tau_2 + \tau_3 =0$, and $\xi_2$ is restricted to $|\xi_2| = |T| + O(1)$. The $\tau_1$-regions are essentially disjoint for $T \in {\mathbb Z}$ and similarly the $\tau_2$-regions. Thus by Schur's test (\cite{T}, Lemma 3.11) we only have to show
\begin{align*}
 &\sup_{T \in {\mathbb Z}} \int_* \frac{\langle \xi_1 \rangle^{\frac{4}{5}-\epsilon} \chi_{\tau_1=-T+O(1)} \chi_{\tau_2=T+O(1)} \chi_{|\tau_3|\sim 1} \chi_{|\xi_2|=|T|+O(1)}}{\langle T \rangle^{\frac{4}{5}-} \langle \xi_3 \rangle^{1+}}\cdot \\
 & \hspace{14em} \cdot\widehat{u}_1(\xi_1,\tau_1) \widehat{u}_2(\xi_2,\tau_2)
\widehat{u}_3(\xi_3,\tau_3) d\xi d\tau \lesssim \prod_{i=1}^3 \|u_i\|_{L^2_{xt}} \, . 
\end{align*}
The $\tau$-behaviour of the integral is now trivial, thus we reduce to
$$ \sup_T \int_{\sum_{i=1}^3 \xi_i =0} m^*(\xi_1,\xi_2,\xi_3,T) \widehat{f}_1(\xi_1)\widehat{f}_2(\xi_2)\widehat{f}_2(\xi_3)d\xi \lesssim \prod_{i=1}^3 \|f_i\|_{L^2_x} $$
with
$$ m^*(\xi_1,\xi_2,\xi_3,T) = \frac{\langle \xi_1 \rangle^{\frac{4}{5}-\epsilon} \chi_{|\xi_2|=|T|+O(1)}}{\langle T \rangle^{\frac{4}{5}-} \langle \xi_3 \rangle^{1+}} \, . $$
It only remains to consider the following two cases: \\
Case 1.1: $|\xi_1| \sim |\xi_3| \gtrsim T$. We obtain in this case
$$m^*(\xi_1,\xi_2,\xi_3,T) \lesssim \frac{ \chi_{|\xi_2|=|T|+O(1)}}{\langle T \rangle^{1+\epsilon-}} \, . $$
Now we estimate as follows
\begin{align*}
&\frac{1}{\langle T \rangle^{1+\epsilon-}} \int_{\sum_{i=1}^3 \xi_i =0}  \chi_{|\xi_2|=T+O(1)} \widehat{f}_1(\xi_1)\widehat{f}_2(\xi_2)\widehat{f}_3(\xi_3) d\xi \\
&\lesssim \frac{1}{\langle T \rangle^{1+\epsilon-}} \|f_1\|_{L^2} \|f_3\|_{L^2} \| {\mathcal F}^{-1}(\chi_{|\xi|=T+O(1)} \widehat{f}_2)\|_{L^{\infty}({\mathbb R}^3)} \\
&\lesssim \hspace{-0.1em} \frac{1}{\langle T \rangle^{1+\epsilon-}} \|f_1\|_{L^2} \|f_3\|_{L^2} \| \chi_{|\xi|=T+O(1)} \widehat{f}_2\|_{L^1({\mathbb R}^3)}\hspace{-0.1em}\lesssim \hspace{-0.1em}\frac{T}{\langle T \rangle^{1+\epsilon-}}  \prod_{i=1}^3 \|f_i\|_{L^2} \hspace{-0.1em}\lesssim\hspace{-0.1em}
\prod_{i=1}^3 \|f_i\|_{L^2}.
\end{align*}
Case 1.2: $|\xi_1| \sim T \gtrsim |\xi_3|$. We get
$$m^*(\xi_1,\xi_2,\xi_3,T) \lesssim \frac{ \chi_{|\xi_2|=|T|+O(1)}}{\langle \xi_3 \rangle^{1+}} \, . $$
An elementary calculation shows that
\begin{align*}
&\Big| \int_{\sum_{i=1}^3 \xi_i =0}  \widehat{f}_1(\xi_1)\widehat{f}_2(\xi_2) \chi_{|\xi_2| = T+O(1)} \widehat{f}_3(\xi_3) \langle \xi_3 \rangle^{-1-} d\xi \Big| \\
&\hspace{2em}\lesssim \| \chi_{|\xi|=T+O(1)} \ast \langle \xi \rangle^{-2-}\|^{\frac{1}{2}}_{L^{\infty}(\mathbb{R}^3)} \prod_{i=1}^3 \|f_i\|_{L^2_x} \lesssim \prod_{i=1}^3 \|f_i\|_{L^2_x}\, ,
\end{align*}
so that the desired estimate follows.\\
{\bf Case 2:} $|\xi_1| \le |\xi_2|$ $\Rightarrow$ $|\xi_2|+|\xi_3| \lesssim |\xi_2|$. \\
Arguing as in case 1 we obtain
$$m^*(\xi_1,\xi_2,\xi_3,T) = \frac{\langle \xi_2 \rangle^{\frac{1}{5}+} \chi_{|\xi_2|=|T|+O(1)}}{\langle \xi_1 \rangle^{\frac{1}{5}+\epsilon} \langle \xi_3 \rangle^{1+}} \, . $$
Case 2.1: $\langle \xi_3 \rangle \le \langle \xi_1 \rangle$.
This implies $\langle \xi_1 \rangle \gtrsim T$, because $|\xi_2| \gtrsim T$ and $|\xi_2| \sim |\xi_1|$. We obtain the same bound for $m^*$ as in case 1.2. \\
Case 2.2: $\langle \xi_1 \rangle \le \langle \xi_3 \rangle$.
Thus $\langle \xi_3 \rangle \gtrsim T$ and we obtain
$$ m^*(\xi_1,\xi_2,\xi_3,T) \lesssim \frac{\langle \xi_2 \rangle^{\frac{1}{5}+} \chi_{|\xi_2|=T+O(1)}}
{\langle \xi_1 \rangle^{\frac{1}{5}+\epsilon} \langle \xi_1 \rangle^{\frac{4}{5}-\epsilon+} \langle \xi_3 \rangle^{\frac{1}{5}+\epsilon}        }
 \lesssim \frac{\chi_{|\xi_2|=T+O(1)}}{\langle T \rangle^{\epsilon -} \langle \xi_1 \rangle^{1+}} \, . $$
In the same way as in case 1.2 we obtain the desired estimate.\\
{\bf Term 3:} Finally we have to consider the cubic terms in (\ref{1.13}).\\
1. We obtain as in (\ref{2.1}) the estimate
$$ \| A^{df} A^{df} \phi\|_{X^{-\frac{1}{4}+,-\frac{1}{4}+}_{|\tau|=|\xi|}} \lesssim \|A^{df}\|^2_{X^{\frac{4}{5}-,\frac{1}{2}+}_{|\tau|=|\xi|}} \|\phi\|_{X^{\frac{4}{5}-,\frac{1}{2}+}_{|\tau|=|\xi|}} \, . $$
2. As in (\ref{2.2}) we obtain
$$ \| A^{cf} A^{df} \phi\|_{X^{-\frac{1}{4}+,-\frac{1}{4}+}_{|\tau|=|\xi|}} \lesssim \|A^{cf}\|_{X^{1+,\frac{1}{2}+}_{\tau=0}} \|A^{df}\|_{X^{\frac{4}{5}-,\frac{1}{2}+}_{|\tau|=|\xi|}} \|\phi\|_{X^{\frac{4}{5}-,\frac{1}{2}+}_{|\tau|=|\xi|}} \, . $$
3. By Sobolev's and Strichartz' estimates (Lemma \ref{Lemma1}) we obtain
\begin{align*}
\| A^{cf} A^{cf} \phi\|_{X^{-\frac{1}{4}+,-\frac{1}{4}+}_{|\tau|=|\xi|}} &\lesssim \|A^{cf} A^{cf} \phi \|_{L^2_{xt}} \lesssim \|A^{cf}\|^2_{L^8_t L^6_x} \|\phi\|_{L^4_t L^6_x} \\
&\lesssim \|A^{cf}\|_{X^{1,\frac{3}{8}}_{\tau=0}}^2 \|\phi\|_{L^4_t H^{\frac{1}{4},4}_x} \lesssim  \|A^{cf}\|_{X^{1,\frac{3}{8}}_{\tau=0}}^2  \|\phi\|_{X^{\frac{3}{4},\frac{1}{2}+}_{|\tau|=|\xi|}} \, .
\end{align*}
Using claim 3 und claim 4
we complete the proof of claim 6 and of Prop. \ref{Prop.1}. As explained in section 2 this also proves Theorem \ref{Theorem}.


\begin{thebibliography}{999999}
\bibitem[AFS]{AFS} P. d'Ancona, D. Foschi, and S. Selberg: {\sl Atlas of products for wave-Sobolev spaces on ${\mathbb R}^{1+3}$ }. Transact. AMS 364 (2012), 31-63
\bibitem[B]{B} M. Beals: {\sl Self-spreading of singularities for solutions to semilinear wave equations}. Annals Math. 118 (1983), 187-214
\bibitem[KRT]{KRT} M. Keel, T. Roy and T. Tao: {\sl Global well-posedness of the Maxwell-Klein-Gordon equation below the energy norm}. Discrete Cont. Dyn. Syst. 30 (2011), 573-621
\bibitem[KMBT]{KMBT} S. Klainerman and M. Machedon (Appendices by J. Bougain and  D. Tataru): {\sl Remark on Strichartz-type inequalities}. Int. Math. Res. Notices 1996, no.5, 201-220
\bibitem[KM]{KM} S. Klainerman and M. Machedon: {\sl On the Maxwell-Klein-Gordon equation with finite energy}. Duke Math. J. 74 (1994), 19-44
\bibitem[MS]{MS} M. Machedon and J. Sterbenz: {\sl Almost optimal local well-posedness for the (3+1)-dimensional Maxwell-Klein-Gordon equations}. J. AMS 17 (2004), 297-359
\bibitem[P1]{P1} H. Pecher: {\sl Nonlinear small data scattering for the wave and Klein-Gordon equation}. Math. Z. 185 (1984), 261-270
\bibitem[P]{P} H. Pecher: {\sl Low regularity local well-posedness for the Maxwell-Klein-Gordon equations in Lorenz gauge}. Adv. Diff. Equ. 19 (2014), 359-386
\bibitem[S]{S} S. Selberg: {\sl Anisotropic bilinear $L^2$ estimates related to the 3D wave equation}. Int. Math. Res. Not. 2008, art. ID rnn107
\bibitem[ST]{ST} S. Selberg and A. Tesfahun: {\sl Finite energy global well-posedness of the Maxwell-Klein-Gordon system in Lorenz gauge}. Comm. PDE 35 (2010), 1029-1057
\bibitem[T]{T} T. Tao: {\sl Multilinear weighted convolutions of $L^2$-functions and applications to non-linear dispersive equations}. Amer. J. Math. 123 (2001), 838-908
\bibitem[T1]{T1} T. Tao: {\sl Local well-posedness of the Yang-Mills equation in the temporal gauge below the energy norm}. J. Diff. Equ. 189 (2003), 366-382
\bibitem[Y]{Y} J. Yuan: {\sl Global solutions of two coupled Maxwell systems in the temporal gauge}. Preprint arXiv:1504.00330
\end{thebibliography}
\end{document}